\documentclass[reqno]{amsart}

\makeatletter
\renewcommand\part{%
	\if@noskipsec \leavevmode \fi
	\par
	\addvspace{4ex}%
	\@afterindentfalse
	\secdef\@part\@spart}

\def\@part[#1]#2{%
	\ifnum \c@secnumdepth >\m@ne
	\refstepcounter{part}%
	\addcontentsline{toc}{part}{\thepart\hspace{1em}#1}%
	\else
	\addcontentsline{toc}{part}{#1}%
	\fi
	{\parindent \z@ \raggedright
		\interlinepenalty \@M
		\normalfontsymmetric decreasing rearrangement
		\ifnum \c@secnumdepth >\m@ne
		\Large\bfseries \partname\nobreakspace\thepart
		\par\nobreak
		\fi
		\huge \bfseries #2%
		\par}%
	\nobreak
	\vskip 3ex
	\@afterheading}
\def\@spart#1{%
	{\parindent \z@ \raggedright
		\interlinepenalty \@M
		\normalfont
		\huge \bfseries #1\par}%
	\nobreak
	\vskip 3ex
	\@afterheading}
\makeatother
\usepackage{color}
\usepackage[dvipsnames]{xcolor}
\usepackage{ifpdf}
\ifpdf 
    \usepackage[pdftex]{graphicx}   
    \usepackage[pdftex,     
            plainpages=false,   
            breaklinks=true,    
            colorlinks=true,
            linkcolor=red,
            citecolor=green,
            pdftitle=My Document
            pdfauthor=My Good Self
           ]{hyperref} 
\else 
    \usepackage{graphicx}       
\fi 

\usepackage{subfig}

\usepackage{glossaries}
\usepackage{glossary-mcols}

\usepackage{aurical}
\usepackage{amsfonts,amsmath}	
\usepackage{amssymb}
\usepackage{verbatim}
\usepackage{amsopn}
\usepackage[english]{babel}
\usepackage{amsthm}
\usepackage{enumerate}
\usepackage{mathrsfs}	
\usepackage{enumitem}
\usepackage{mathtools}
\usepackage{esint}
\usepackage{bbm}
\usepackage{caption}
\usepackage{marginnote}
\usepackage[marginparwidth=2cm]{geometry}
\usepackage{csquotes}
\usepackage{cleveref}
\crefname{enumi}{part}{parts}
\date{\today}

\theoremstyle{definition} \newtheorem{definition}{Definition}[section]
\theoremstyle{definition} \newtheorem{remark}[definition]{Remark}
\theoremstyle{plain} \newtheorem{lemma}[definition]{Lemma}
\theoremstyle{plain} \newtheorem{proposition}[definition]{Proposition}
\theoremstyle{plain} \newtheorem{theorem}[definition]{Theorem}
\theoremstyle{plain} \newtheorem{corollary}[definition]{Corollary}
\theoremstyle{definition} 
\theoremstyle{plain} 
\theoremstyle{definition} 
\theoremstyle{definition}

\DeclareMathOperator{\dive}{div}

\newcommand{\norm}[1]{\lVert #1\rVert}

\newcommand{\R}{\mathbb{R}}

\newcommand{\Per}{\mathbf{Per}}

\numberwithin{equation}{section} 

\theoremstyle{plain} \newtheorem*{theorem*}{Theorem}
\theoremstyle{plain} 
\theoremstyle{plain} \newtheorem*{mthm*}{Main Theorem}
\theoremstyle{plain} \newtheorem*{conjecture*}{Conjecture}
\theoremstyle{plain} 
\theoremstyle{plain} \newtheorem*{problem*}{Problem}

\usepackage{biblatex} 
\title{Existence and Blow-Up for Non-linear Fokker-Planck with Controlled Drift Divergence}

\author{Giacomo Maria Leccese}
\email{giacomo.leccese94@gmail.com}

\begin{document}

\begin{abstract}
We extend the global existence results for critical parameters of \cite{bianchini2024existence} to the multidimensional problem
\begin{equation*}
    \partial_t u + \dive (b(t,x) u^{1+k}) = \Delta u
\end{equation*}
for $b:\R\times\R^d\to\R^d$ non-autonomous fields, in the case
\begin{equation*}
   (\dive b)_-\in L^\infty_{\mathrm{loc}}
  \bigl([0,\infty);L^{p,\infty}(\R^d)\bigr),
  \qquad p>d.
\end{equation*}
The proof of global existence follows the parabolic symmetrization approach of \cite{bandle1976symmetrizations} and is  based on the study of the function 
\begin{equation*}
  m(t,s)=\int_0^s u^*(t,r)\,dr.
\end{equation*}

\end{abstract}

\maketitle

\section{Introduction}

In this paper,  we study the global-in-time existence and the long-time behaviour of non-negative solutions to the nonlinear Fokker-Planck equation
\begin{equation}\label{eq:pde}
\left\{
\begin{aligned}
  &\partial_t u+\dive\bigl(b(t,x)u^{1+k}\bigr)=\Delta u,
    &&(t,x)\in(0,\infty)\times\R^d,\\
  &u(0,x)=u_0(x),
    &&x\in\R^d,
\end{aligned}
\right.
\end{equation}
where $k>0$, the initial datum is non-negative and
\begin{equation*}
  u_0\in L^1(\R^d)\cap L^\infty(\R^d).
\end{equation*}
Equation \eqref{eq:pde} describes the competition between linear diffusion and nonlinear drift. Since the equation is in divergence form, non-negative bounded solutions conserve the mass, i.e.
\begin{equation*}
  \norm{u(t)}_1=\norm{u_0}_1.
\end{equation*}
We assume that
\begin{equation}\label{eq:assumptions}
\begin{split}
& k>0, \ p>d\\
&  b\in L^\infty_{loc}\bigl((0,\infty)\times\R^d\bigr),\\
&(\dive b)_-\in L^\infty_{loc}
  \bigl((0,\infty),L^{p,\infty}(\R^d)\bigr).
\end{split}
\end{equation}
The local boundedness of $b$ is only used to obtain local existence and uniqueness, whereas the global estimate depends entirely on the negative part of its divergence bound.

\medskip
Nonlinear Fokker-Planck equations with superlinear drift arise in
kinetic theory and statistical mechanics. An important family was
introduced by Kaniadakis-Quarati to describe the kinetic evolution of fermionic and bosonic particle systems, leading respectively to
Fermi-Dirac and Bose-Einstein equilibrium distributions
\cite{kaniadakis1994classical,kaniadakis1993kinetic}.

Further results on the associated fermionic and bosonic equations,
including convergence to equilibrium and finite-time concentration,
can be found in
\cite{carrillo20081d,toscani2012finite}.
Related questions concerning blow-up and large-time behaviour for
nonlinear diffusion-advection and convection-diffusion equations
were studied in
\cite{alikakos1989blow,escobedo1991large, kaplan1963growth,
kieffer2025robust, toscani2025supercritical,}.

\medskip

Equations of the form \eqref{eq:pde} were first studied in dimension
one in \cite{guidolin2022global}, where global existence was obtained
in the subcritical regime for bounded Lipschitz drifts. A complete
classification of the one-dimensional problem was subsequently given
in \cite{bianchini2024existence}, under assumptions either on $b$
or on its derivative $b_x$ in weak Lebesgue spaces.

The scaling associated with a parabolic concentration gives the threshold
\begin{equation}\label{eq:critical-exponent}
  k_c:=\frac2d-\frac1p>0.
\end{equation}
Indeed, under
\begin{equation*}
  u_\lambda(t,x)=\lambda^d u(\lambda^2t,\lambda x),
  \qquad
  b_\lambda(t,x)=\lambda^{1-dk}b(\lambda^2t,\lambda x),
\end{equation*}
one has
\begin{equation*}
  \norm{(\dive b_\lambda(t))_-}_{L^{p,\infty}}
  =\lambda^{2-dk-d/p}
   \norm{(\dive b(\lambda^2t))_-}_{L^{p,\infty}},
\end{equation*}
which is invariant exactly for $k=k_c$.

The present paper extends the divergence-controlled part of this
classification to higher dimensions. The natural multidimensional
analogue of $b_x$ is $\dive b$. A direct application of the one-dimensional argument cannot obtain the multidimensional critical case. Indeed, the truncated entropy method used in \cite{bianchini2024existence} is based on the embedding $H^1(\R)\hookrightarrow L^\infty(\R)$, which fails in higher dimensions. Instead, we adapt the parabolic
symmetrization method introduced by Bandle in
\cite{bandle1976symmetrizations}. More precisely, we study the
mass-concentration function $m(t,s):=\int_0^s u^*(t,r)\ d r,$
where $u^*(t,\cdot)$ is the decreasing rearrangement of
$u(t,\cdot)$ in the volume variable, and compare $m(t,s)$ with an
explicit stationary profile. The derivative of this stationary barrier has the form
\begin{equation}\label{eq:intro-slope}
  f_\lambda(s)=
  \left(\lambda+\frac{k\sup_{0<t<T}\norm{ {(\dive b(t))_-}}_{L^{p,\infty}}}{d^{2-2/d}\omega_d^{2-2/d}k_c}s^{k_c}\right)^{-1/k}.
\end{equation}
When $k=k_c$, the singular limiting profile behaves like $s^{-1}$, so its primitive diverges logarithmically near the origin. Hence a barrier can contain arbitrarily large mass in any prescribed positive volume.  When
$k<k_c$, the singularity is stronger and the same conclusion holds. When $k>k_c$, the limiting profile is locally integrable, and this mechanism breaks down. The results of this paper on
the global existence can be summarised in the following theorem.

\begin{theorem}\label{thm:global-main}
Assume conditions \eqref{eq:assumptions}.  If $0<k\le k_c$,
then the solution of \eqref{eq:pde} is globally defined.
More precisely, for every $T>0$ there is a finite constant $\bar\lambda_T$ (defined in equation \eqref{eq:uniformconst}),
depending only on $d,p,k,\norm{u_0}_1,\norm{u_0}_\infty$ and
$
  \sup_{t\in[0,T]}\norm{ {(\dive b(t))_-}}_{L^{p,\infty}},
$
such that
\begin{equation*}
  \sup_{0\leq t\leq T}\norm{u(t)}_\infty\leq \bar\lambda_T.
\end{equation*}
\end{theorem}

In the supercritical domain, the energy method of \cite{ bianchini2024existence,toscani2012finite} can be generalised for the multidimensional problem as follows.

\begin{theorem}\label{thm:blowup-main}
Assume conditions \eqref{eq:assumptions}. If
$k> \frac2d-\frac1p$, there exists a time-independent radial vector field $b$, with
\begin{equation*}
  \dive b\in L^p(\R^d)\subset L^{p,\infty}(\R^d),
\end{equation*}
and a non-negative
\begin{equation*}
  u_0\in L^1(\R^d)\cap L^\infty(\R^d),
  \qquad
  \int_{\R^d}|x|^2u_0(x)\,d x<\infty,
\end{equation*}
for which the corresponding solution blows up in finite time in the
$L^\infty$-norm.
\end{theorem}

For the long-time behaviour of the solutions, we consider the assumptions
\begin{equation}\label{eq:assumptions2}
\begin{split}
& k>0, \ p>d,\\
&(\dive b)_-\in L^\infty
  \bigl((0,\infty);L^{p,\infty}(\R^d)\bigr).
\end{split}
\end{equation}

The scenario for the long-time behaviour is described in the following theorem.
\begin{theorem}
\label{Theop:longtime_1}
Assume condition \eqref{eq:assumptions2}. If $k>k_c$, any uniformly bounded solution $u$ in space and time, i.e. $u \in L^\infty_{t,x}$, satisfies
\begin{equation}\label{eq:longtimeest}
\max_{x\in\R^d}|u(t,x)|\le  \mathcal O(t^{-\frac{d}{2}}), \text{ for }t \to \infty.
\end{equation} If $k=k_c$, any well defined solution satisfies \eqref{eq:longtimeest}. Instead, for subcritical cases $k<{k_c}$, there are examples of stationary solutions.
\end{theorem}
The paper is organised as follows. In \Cref{sec:preliminaries}, we introduce the notation, recall Lorentz spaces, and prove some preliminary lemmas. In \Cref{sec:local}, we state the local well-posedness of the solution. \Cref{sec:globex} is dedicated to the proof of global existence. The supercritical blow-up construction is given in \Cref{sec:blowup}. In \Cref{sec:longbehav}, we study the decay behaviour of the solution.

\section{Notation and Preliminaries}\label{sec:preliminaries}

We will use the following notation.
\begin{itemize}
	\item For every positive function $w\in L^1_{loc}(\R^d)$, we denote\begin{itemize}
	\item $\mu_w(\lambda)=|\{w>\lambda\}|$,
    \item $m_w(s)=\int_0^s w^*(r)\ dr$,
    \item $E(w)=\int |x|^2 w(x) \ dx$,
    \item $w_-=\max\{-w,0\}$, $w_+=\max\{w,0\}$,
    \item $D^-\mu_w(s)=\limsup_{h\to0^+} \frac{\mu_w(s)-\mu_w(s-h)}{h}$ the upper left derivative.
	\end{itemize}
	\item $\omega_d$ is the surface area of the $(d-1)$-sphere of radius 1.
	\item For some number $p\in[1,\infty]$, we denote $p'=\frac{p}{p-1}$ the conjugate of $p$.
	\item The letter $C$ will denote some constants that could change line by line. Any other constant will be uniquely fixed.
	\item For every function $v$ in two variables $t,x$, we denote by $v(t)$ the function $x\mapsto v(t,x)$. 
\end{itemize}
\subsection{Lorentz spaces}
\label{sec:lorentz}

We briefly recall the definition and some well known results about the Lorentz space. For any $w:\R^d\to \R$, let $w^*:[0,\infty)\to \R$ be the symmetric decreasing rearrangement defined by
\begin{equation*}
	w^{\ast}(z) = \inf \Big\{\alpha>0: \mathcal L^d\big(\big\{|w|>\alpha\big\}\big)|\leq z\Big\}.
\end{equation*}
Define 
\begin{equation}
	\label{Equa:Lorentz_2}
	\|w\|_{p,q} = \begin{cases}
		\left( \displaystyle \int_0^{\infty} \big[ z^{\frac{1}{p}} w^{*}(z) \big]^q \, \frac{dz}{z} \right)^{\frac{1}{q}} & q \in [1, \infty), p \in [1,\infty), \\
		\sup\limits_{z > 0} \, z^{\frac{1}{p}}  w^{*}(z)   & q = \infty, p \in [1,\infty).
	\end{cases}
\end{equation}
with the notation $\|w\|_{\infty,\infty}=\|w\|_{\infty}$.
\begin{proposition}
	$L^{p}$ is embedded in $L^{p,\infty}$. In particular for every function $w:\R^d \to \R$,
	$$\|w\|_{p,\infty}\le \|w\|_p$$
\end{proposition}
See \cite[Proposition 1.1.6]{grafakos2008classical}
\begin{theorem}[Hardy–Littlewood inequality]
	Given $w_1,w_2$ non-negative functions, it holds
	$$
	\int _{\mathbb {R} ^{n}}w_1(x)w_2(x)\,dx\leq \int _{\mathbb {R} ^{n}}w_1^{*}(z)w_2^{*}(z)\,dz.$$
\end{theorem}
See \cite{hardy1952inequalities}.
\begin{theorem}[H\"older’s inequality in Lorentz spaces]
	\label{Theo:holder_lorentz}
	Let $p\in[1,\infty)$ $ p_1\in [1,\infty), p_2 \in [1,\infty]$ such that 
	$$\frac{1}{p}=\frac{1}{p_1}+\frac{1}{p_2},$$
	then 
	$$
	\|w_1w_2\|_{p} \le  \|w_1\|_{p_1,p} \|w_2\|_{p_2,\infty}.
	$$
\end{theorem} 

\begin{proof}
	The case $p_2=\infty$ is trivial. For $p_2<\infty$, by Hardy–Littlewood inequality
	\begin{equation*}
		\begin{split}
			\|w_1w_2\|_p^p&\le\int(|w_1|^p)^{*}(|w_2|^p)^{*}\\
			&=\int(w_1^{*})^p(w_2^{*})^p\\
			&\le \sup z^\frac{p}{p_2}(w_2^{*})^p\int (z^\frac{1}{p_1}f^*)^p\frac{dz}{z}\\
			&=\|w_1\|_{p_1,p}^p \|w_2\|_{p_2,\infty}^p,
		\end{split}
	\end{equation*}
	where we used $\left(|w|^{p}\right)^{*}=(w^{*})^{p}$, that is easy to prove.
\end{proof}
\begin{theorem}[Young inequality in Lorentz spaces]
\label{Theo:young_lorentz}
Let $p,p_1\in(1,\infty)$, $p_2\in[1,\infty)$ such that 
	$$\frac{1}{p}+1=\frac{1}{p_1}+\frac{1}{p_2},$$
	then 
	$$
	\|w_1\star w_2\|_{p,\infty} \le  C\|w_1\|_{p_1,\infty} \|w_2\|_{p_2}.
	$$
\end{theorem} 
See, for instance, \cite[Theorem 1.4.25]{grafakos2008classical}.
\subsection{Symmetric Decreasing Rearrangement lemmas}
We prove some well-known facts about Symmetric Decreasing Rearrangement lemmas that will be useful in the paper
\begin{lemma}
\label{lem:exact-set}
Let $w$ non-negative integrable function, and fix $s>0$. It holds
\begin{equation*}
    m_w(s):=\int_0^s w^*(r)\, d r=\max_{\substack{A\subset\R^d\\ |A|=s}}
    \int_A w(x)\,d x
\end{equation*}
and a measurable set $A\subset\R^d$ with $|A|=s$
is a maximizer if and
only if, up to negligible sets,
\begin{equation*}
    \{w>w^*(s)\}
    \subseteq A
    \subseteq
    \{w\ge w^*(s)\}. 
\end{equation*}
In particular, if $|\{w=w^*(s)\}|=0$, then one can choose
\begin{equation*}
    \int_0^s w^*(r)\, dr= \int_{\{w>w^*(s)\}} w (x)\,d x, 
\end{equation*}
\end{lemma}

\begin{proof} For every measurable set $A\subset\R^d$ with $|A|=s$, the Hardy-Littlewood inequality gives
\begin{equation*}
    \int_A w
    =
    \int_{\R^d}w1_A
    \leq
    \int_0^\infty w^*(r)(1_A)^*(r)\,d r=\int_0^s w^*(r)\,d r,
\end{equation*}
since $( 1_A)^*= 1_{(0,s)}$.
On the other hand, if $|A|=s$, it holds
\begin{equation*}
    \int_0^s w^*=\int_{\{w>w^*(s)\}}(w-w^*(s))\,d x+w^*(s)s
\end{equation*}
and 
\begin{equation*}
    \int_A w=\int_{A\cap \{w>w^*(s)\}}(w-w^*(s))\,d x-\int_{A\cap \{w<w^*(s)\}}(w-w^*(s))\,d x+w^*(s)s
\end{equation*}
so that
\begin{equation*}
\begin{aligned}
    \int_0^s w^*-\int_A w
    &=\int_{\{w>w^*(s)\}\setminus A} (w-w^*(s))\,d x +\int_{A\cap\{w<w^*(s)\}}(w^*(s)-w)\,d x.
\end{aligned}
\end{equation*}
 The equality holds if and only if
\begin{equation*}
    \{w>w^*(s)\}\subseteq A
    \subseteq\{w\geq w^*(s)\}
\end{equation*}
up to negligible sets. This concludes the proof.
\end{proof}
We remind the notation $\mu_w(\lambda):= |\{w>\lambda\}|$. 
\begin{lemma}\label{lem:derivative_coarea}
For every $w\in C^1(\R^d)$, for the upper left derivative of $\mu$, for a.e. $\lambda$, the following holds
\begin{equation}\label{eq:distribution-derivative}
 D^-\mu_w(\lambda)\le-\int_{\{w=\lambda\}}\frac1{|\nabla w|}\,d\mathcal H^{d-1}.
\end{equation}
\end{lemma}
\begin{proof}
For $a<b$, the coarea formula gives
\begin{equation*}
  \mu_w(a)-\mu_w(b)\ge\{a<w\le b,\ |\nabla w|>0\}
  =\int_a^{b}
  \left(\int_{\{w=r\}}\frac1{|\nabla w|}
  \,d\mathcal H^{d-1}\right)d r.
\end{equation*}
\end{proof}

\begin{proposition}
For almost every s, it holds
\begin{equation}\label{eq:Mss-level}
	\frac d {ds}w^*(s)\ge-\frac1{\int_{\{w=w^*(s)\}}\frac1{|\nabla w|}\,d\mathcal H^{d-1}}.
\end{equation}
\end{proposition}
\begin{proof}
Since $w^*$ is decreasing, it is differentiable a.e. At any point where $D [w^*](s)=0$, the result is immediate. Consider $s$ such that $D [w^*](s)<0$. We claim that 
\begin{equation}\label{eq:inverseidentity}
    \mu_w(w^*(s))=s.
\end{equation}
By definition of symmetric decreasing rearrangement, it holds $\mu_w(w^*(s))\le s$. If by contradiction, $\mu_w(w^*(s))<s$, then, by the definition of the generalized inverse,
$w^*(r)=w^*(s)$, for every $r\in(\mu_w(w^*(s)),s)$. Hence $D^-[w^*](s)=0$, that is a contradiction.

For $h>0$, let $\delta_h:=w^*(s)-w^*(s+h)>0.$ For equation \eqref{eq:inverseidentity}
\begin{equation*}
D^-\mu_w(w^*(s))\ge\lim_{h \to 0}\frac{\mu_w(w^*(s))-\mu_w(w^*(s)-\delta_h)}{\delta_h}
\ge
-\lim_{h \to 0^+}\frac{h}{\delta_h}=\frac 1{d/ds \,w^*(s)}
\end{equation*}
The thesis follows by Lemma \ref{lem:derivative_coarea}.
\end{proof}

\subsection{Stampacchia's property} We remind a corollary of Stampacchia's property that will be useful later.
\begin{lemma}\label{lemma:stamp}
    For any $w\in W^{1,r}$, $r\ge1$ it holds
    \begin{equation*}
        \nabla u=0 \qquad\text{for a.e. }x\in\{u=0\},
    \end{equation*}
    and for any $w\in W^{2,r}$, $r\ge1$
    \begin{equation*}
        D^2u=0 \qquad\text{for a.e. }x\in\{u=0\}.
    \end{equation*}
\end{lemma}
See \cite[Appendix C, Lemma 4]{grunau2024biharmonic}
\section{Local existence and continuation}\label{sec:local}
The purpose of this section is to demonstrate the well-posedness and non-negativity of the solution, as well as to justify certain regularity assumptions that will prove useful later in the article.
\subsection{Local well-posedness}
We prove the local well-posedness of the solution. Let
\begin{equation*}
  G(t,x)=\frac1{(4\pi t)^{d/2}}
  \exp\left(-\frac{|x|^2}{4t}\right)
\end{equation*}
be the heat kernel.

\begin{proposition}[Local bounded solutions]\label{prop:local}
There is a maximal time $T_*\in(0,\infty]$ and a unique solution
\begin{equation*}
  u\in C([0,T_*);L^1(\R^d))
  \cap L^\infty_{\mathrm{loc}}
  ([0,T_*)\times\R^d)
\end{equation*}
satisfying
\begin{equation}\label{eq:duhamel}
  u(t)=G(t)*u_0-
  \int_0^t\nabla G(t-s)*\bigl(b(s)u(s)^{1+k}\bigr)\,d s.
\end{equation}
The solution conserves the mass.  
\end{proposition}

\begin{proof}
The proof holds by showing that the Duhamel's formula \eqref{eq:duhamel} describe the fixed point of the contraction operator
$$\Phi[u](t)=G(t)*u_0-
  \int_0^t\nabla G(t-s)*\bigl(b(s)u(s)^{1+k}\bigr)\,d s
$$ in norm $L^\infty\cap L^1$ on the set 
$$
S = \left\{ u \in {C([0,T], L^1(\R^d))\cap L^\infty([0,T],L^\infty(\R^d))}: \|u\|_\infty \le r,\, u(0) = u_0 \right\}.
$$
First, $\Phi[u](t)$ is a continuous function, indeed
\begin{equation*}
		\begin{split}
			\|\Phi[u](t+\delta) - \Phi[u](t)\|_\infty &\le \|G(t+\delta) - G(t)\|_1 \|u_0\|_\infty  \\
     &\qquad  +\int_t^{t+\delta} \|{ G_x(t+\delta -s)}\|_{1} \|b(s) u^{1+k}(s)\|_{\infty} ds \\
			& \qquad + \int_0^t \|{ G_x(t + \delta - s) - G_x(t-s)}\|_{1} \|b(s) u^{1+k}(s)\|_{\infty} ds \\
		&\leq \|G(t+\delta) - G(t)\|_1 \|u_0\|_\infty \\
    &\qquad + C \|b\|_{\infty} \|u\|_\infty^{k+1} \bigg( \delta^{\frac{1}{2}} + \int_0^t \|{  G_x(t + \delta - s) - G_x(t-s)}\|_{1} ds \bigg).
		\end{split}
	\end{equation*}
The last terms converge to $0$ as $\delta \to 0$ uniformly in every interval of the form $[t_0,T]$, $t_0>0$.
The uniform continuity in norm $L^1$ norm on [0,T], follows analogously using 
    \begin{equation*}
\begin{split}
			\|\Phi[u](t+\delta) - \Phi[u](t)\|_1 &\leq \|G(t+\delta) \ast u_0 - G(t)\ast u_0\|_1 \\
      &\qquad  +\int_t^{t+\delta} \|{ G_x(t+\delta -s)}\|_{1} \|b(s) u^{1+k}(s)\|_{1} ds \\
			& \qquad + \int_0^t \|{ G_x(t + \delta - s) - G_x(t-s)}\|_{1} \|b(s) u^{1+k}(s)\|_{1} ds \\
			&\leq \|G(t+\delta) - G(t)\|_1 \|u_0\|_\infty \\
    &\qquad + C \|b\|_{\infty} \|u\|_\infty^{k}\|u\|_1 \bigg( \delta^{\frac{1}{2}} + \int_0^t \|{  G_x(t + \delta - s) - G_x(t-s)}\|_{1} ds \bigg).
\end{split}
	\end{equation*}
Similarly, it is easy to prove that $\Phi[u](t)$ is a contraction in norm $L^1$ and in norm $L^\infty$, using 
\begin{equation*}
  \norm{\nabla G(t-s)*\bigl(bu^{1+k}\bigr)}_\infty
  \leq C(t-s)^{-1/2}\norm{b}_\infty  \|u\|_\infty^{1+k},
\end{equation*}
and
\begin{equation*}
  \norm{\nabla G(t-s)*\bigl(bu^{1+k}\bigr)}_1
  \leq C(t-s)^{-1/2}\norm{b}_\infty  \|u\|_\infty^k\norm{u}_1.
\end{equation*}
For $T\ll1$ it holds $\Phi(S)\subseteq S$. 
Finally for Duhamel formula, it is easy to prove that the solution preserve the mass, by using 
\begin{equation*}
    \int_{\mathbb R^d}
    \nabla G(t-s)\star\big(b_i(s)u(s)^{1+k}\big)\,dx
    =\int_{\mathbb R^d}b_i(s,y)u(s,y)^{1+k}\int_{\mathbb R^d}\nabla G(t-s,x-y)\,dx\,dy \\
    =0,
\end{equation*}
This concludes the thesis. 
\end{proof}

\begin{remark}
Alternatively, without assuming $b$ bounded, one could impose $E(u_0)$ and $\|D b\|_{p,\infty}$ bounded and use the formula $$
|b(t,x)-b(t,y)|\le C\|D b(t)\|_{p,\infty}|x-y|^\alpha,$$
and follows the same estimates of \cite{bianchini2024existence}.
\end{remark}
\begin{corollary}\label{prop:aprioriest}
There is a time $\tau$ and $C$, both depending on $\norm{b}_\infty,\norm{u_0}_1,\norm{u_0}_\infty$, such that any solution of Proposition \ref{prop:local} satisfies
\begin{equation*}
    \sup_{t \in [0,\tau]}\norm{u}_\infty\le C
\end{equation*}
\end{corollary}
\begin{proof}
The time $\tau>0$ can be selected by imposing that the operator $\Phi$ in Proposition \ref{prop:local} is a contraction, and in this case, for similar computation of before, it holds
\begin{equation*}
  \norm{u(t)}_\infty\le \|u_0\|_\infty+C\norm{b}_\infty\norm{u_0}_\infty^{k+1}\sqrt t
\end{equation*}  
for every $t\leq \tau$.
\end{proof}
\subsection{Energy regularity estimate}
We prove a standard energy regularity.
\begin{proposition}\label{prop:energy_reg}
  It holds
\begin{equation*}
    u\in L^\infty((0,T),L^2(\mathbb R^d))
    \cap L^2((0,T),H^1(\mathbb R^d)),
    \qquad
    \partial_tu\in L^2((0,T),H^{-1}(\mathbb R^d)).
\end{equation*}
\end{proposition}
\begin{proof}
It is easy to prove that Duhamel formula \eqref{eq:duhamel} implies \eqref{eq:pde} as distribution. Since $u\in L^2$, we can use it as test function to obtain
\begin{equation}\label{eq:energy_est}
\begin{split}
    \frac d{dt}\frac 1 2 \|u(t)\|_2^2
    +\|\nabla u(t)\|_2^2
    &=
    \int_{\mathbb R^d}
    b(t,x)u(t,x)^{1+k}\cdot\nabla u(t,x)\,dx\\
    &\le \frac12\|\nabla u(t)\|_2^2
    +
    \frac12 \|b\|_{L^\infty_{t,x}}^2\|u\|_{L^\infty_{t,x}}^{2k}\|u(t)\|_2^2.
\end{split}
\end{equation}
By Gronwall's inequality and \eqref{eq:energy_est}, it holds
\begin{equation*}
    \|u(t)\|_2^2
    \leq
    \|u_0\|_2^2
    \exp\bigl(\|b\|_{L^\infty_{t,x}}^2\|u\|_{L^\infty_{t,x}}^{2k}t\bigr),
    \qquad 0\leq t\leq T.
\end{equation*}
Integrating \eqref{eq:energy_est} in time also gives
\begin{equation*}
    \int_0^T\|\nabla u(t)\|_2^2\,dt
    \leq
    \|u_0\|_2^2
    +
    \|b\|_{L^\infty_{t,x}}^2\|u\|_{L^\infty_{t,x}}^{2k}
    \int_0^T\|u(t)\|_2^2\,dt
    <\infty.
\end{equation*}
Consequently,
\begin{equation*}
    u\in
    L^\infty((0,T),L^2(\mathbb R^d))
    \cap
    L^2((0,T),H^1(\mathbb R^d)).
\end{equation*}

Finally, for every $\varphi\in H^1(\mathbb R^d)$ test function,
\begin{equation*}
\begin{aligned}
    |\langle\partial_tu,\varphi\rangle|
    &\leq
    \|\nabla u\|_2\|\nabla\varphi\|_2
    +
    \|bu^{1+k}\|_2\|\nabla\varphi\|_2
    \\
    &\leq
    \bigl(
    \|\nabla u\|_2
    +
    \|b\|_{L^\infty_{t,x}}^2\|u\|_{L^\infty_{t,x}}^{2k}\|u\|_2
    \bigr)
    \|\varphi\|_{H^1}.
\end{aligned}
\end{equation*}
Thus
\begin{equation*}
    \|\partial_tu(t)\|_{H^{-1}}
    \leq
    \|\nabla u(t)\|_2
    +
   \|b\|_{L^\infty_{t,x}}^2\|u\|_{L^\infty_{t,x}}^{2k}\|u(t)\|_2,
\end{equation*}
and therefore
\begin{equation*}
    \partial_tu
    \in L^2((0,T),H^{-1}(\mathbb R^d)).
\end{equation*}
\end{proof}
\subsection{Positiveness}
Now we prove that the solution is always non-negative.
\begin{proposition}
  The solution defined in Proposition \ref{prop:local} is non-negative
\end{proposition}
\begin{proof}Define the function
$r \mapsto F(r):=(r_+)^{1+k}$. Since $F$ is locally Lipschitz on $\mathbb R$, the previous contraction
argument applies to
\begin{equation*}
    \partial_t u+\dive\bigl(bF(u)\bigr)=\Delta u.
\end{equation*}
We now prove that the corresponding solution is non-negative, so that
$F(u)=u^{1+k}$ and hence it solves the original equation. Testing the equation with $u_-$ gives
\begin{equation*}
    -\frac12\frac d{dt}\|u_-(t)\|_2^2
    +\int_{\mathbb R^d}
    \dive\bigl(bF(u)\bigr)u_-\,dx
    =
    \int_{\mathbb R^d}\Delta u\,u_-\,dx.
\end{equation*}
Since $F(u)=0$ on $\{u\leq0\}$ and
$\nabla u_-=0$ a.e. on $\{u\geq0\}$,
\begin{equation*}
    \int_{\mathbb R^d}
    \dive\bigl(bF(u)\bigr)u_-\,dx
    =
    -\int_{\mathbb R^d}
    bF(u)\cdot\nabla u_-\,dx
    =0.
\end{equation*}
Moreover,
\begin{equation*}
    \int_{\mathbb R^d}\Delta u\,u_-\,dx
    =
    -\int_{\mathbb R^d}\nabla u\cdot\nabla u_-\,dx
    =
    \|\nabla u_-\|_2^2.
\end{equation*}
Hence
\begin{equation*}
    \frac12\frac d{dt}\|u_-(t)\|_2^2
    +\|\nabla u_-(t)\|_2^2
    =0.
\end{equation*}
This proves that $u=u_+\ge0$.
\end{proof}

\subsection{Reduction to smooth solutions}\label{subsec:smoothreduction} We want to prove that, throughout a standard derivation of a priori estimates and smooth approximations, for the rest of the paper, we may consider the solution to be smooth. We can consider a smooth approximation $u_{0,\varepsilon}$ such that
\begin{equation*}
    u_{0,\varepsilon}\in C^\infty(\mathbb R^d),
\qquad
u_{0,\varepsilon}>0,
\qquad
u_{0,\varepsilon}\longrightarrow u_0
\quad\text{in }L^1(\mathbb R^d),
\end{equation*}
and
\begin{equation*}
\|u_{0,\varepsilon}\|_1=\|u_0\|_1,
\qquad
\|u_{0,\varepsilon}\|_\infty
\leq\|u_0\|_\infty.
\end{equation*}
As well, we can consider a smooth approximation $ b_\varepsilon$ such that
\begin{equation*}
\|b_\varepsilon\|_\infty\leq\|b\|_\infty, \qquad (\dive b_\varepsilon)_-\leq\rho\star(\dive b)_-,
\end{equation*}
where $\rho$ is a smooth mollifier, and by Theorem \ref{Theo:young_lorentz}
\begin{equation*}
\|(\dive b_\varepsilon)_-\|_
{L^\infty_tL^{p,\infty}_x}
\leq
C_p
\|(\dive b)_-\|_
{L^\infty_tL^{p,\infty}_x},
\end{equation*}
where $C_p$ is independent of $\varepsilon$.
Let $u_\varepsilon$ be the maximal solution of
\begin{equation*}
\begin{cases}
\partial_t u_\varepsilon
+\dive\!\left(b_\varepsilon u_\varepsilon^{1+k}\right)
=\Delta u_\varepsilon,
\\
u_\varepsilon(0)=u_{0,\varepsilon}.
\end{cases}
\end{equation*}
Then in any time interval $[0,T]$ where the solution is defined,
it solves the equation
\begin{equation*}
    \partial_t u_\varepsilon-\Delta u_\varepsilon= - \dive F_\varepsilon,
\end{equation*}
where 
\begin{equation*}
F_\varepsilon := b_\varepsilon u_\varepsilon^{1+k} \in L^\infty([0,T]\times\mathbb R^d).
\end{equation*}

The Hölder estimates for heat potentials
\cite[Chapter IV, Sections 1-2]{ladyzhenskai1968linear}
then imply, for every $\alpha\in(0,1)$,
\begin{equation*}
    u_\varepsilon
\in
C^{\alpha/2,\alpha}_{\mathrm{loc}}
\bigl((0,T_\varepsilon)\times\mathbb R^d\bigr).
\end{equation*}

Since $r\mapsto r^{1+k}$ is locally Lipschitz on $[0,\infty)$ 
\begin{equation*}
F_\varepsilon=b_\varepsilon u_\varepsilon^{1+k}
\in
C^{\alpha/2,\,\alpha}_{\mathrm{loc}}.
\end{equation*}
Applying again the Hölder estimates in
\cite[Chapter IV, Section 2]{ladyzhenskai1968linear} it holds
\begin{equation*}
u_\varepsilon
\in
C^{(\alpha+1)/2,\,\alpha+1}_{\mathrm{loc}}
\bigl((0,T_\varepsilon)\times\mathbb R^d\bigr).
\end{equation*}
and thus
\begin{equation*}
u_\varepsilon^k
\in
C^{\gamma/2,\,\gamma}_{\mathrm{loc}}
\bigl((0,T_\varepsilon)\times\mathbb R^d\bigr), \qquad \gamma=\min\{\alpha,k\}
\end{equation*}
Then the solution solves
\begin{equation*}
\partial_tu_\varepsilon-\Delta u_\varepsilon
= f_\epsilon
\end{equation*}

where 
\begin{equation*}
 f_\epsilon=-(1+k)b_\varepsilon u_\varepsilon^k
\cdot\nabla u_\varepsilon
-(\dive b_\varepsilon)u_\varepsilon^{1+k}\in C^{\gamma/2,\gamma}_{\mathrm{loc}}.
\end{equation*}

The Schauder estimate for the heat equation
\cite[Chapter IV, Section 2]{ladyzhenskai1968linear} therefore gives
\begin{equation}\label{eq:ueps-classical}
u_\varepsilon
\in
C^{1+\gamma/2,\,2+\gamma}_{\mathrm{loc}}
\bigl((0,T_\varepsilon)\times\mathbb R^d\bigr).
\end{equation}
In particular, $u_\varepsilon$ is a classical solution.

We can now apply the strong maximum principle. Indeed the solution solves
\begin{equation*}
\partial_tu_\varepsilon-\Delta u_\varepsilon
+a_\varepsilon\cdot\nabla u_\varepsilon
+c_\varepsilon u_\varepsilon=0,
\end{equation*}
where
\begin{equation*}
a_\varepsilon
=(1+k)b_\varepsilon u_\varepsilon^k,
\qquad
c_\varepsilon
=(\dive b_\varepsilon)u_\varepsilon^k.
\end{equation*}
These coefficients are continuous and locally bounded. After an
exponential change of the dependent variable to make the zeroth-order
coefficient non-negative, the parabolic strong maximum principle
\cite[Section 7.1.4, Theorem 12]{evans2022partial} applies.
Since $u_{0,\varepsilon}>0$, it follows that
\begin{equation}\label{eq:ueps-positive}
u_\varepsilon(t,x)>0,
\qquad
t>0,\quad x\in\mathbb R^d.
\end{equation}

Finally, $r\mapsto r^{1+k}$ is smooth far from $0$. Repeated application of the interior
Schauder estimates for the heat equation
\cite[Chapter IV, Section 2]{ladyzhenskai1968linear} then gives
\begin{equation}\label{eq:ueps-smooth}
u_\varepsilon
\in
C^\infty_{\mathrm{loc}}
\bigl((0,T_\varepsilon)\times\mathbb R^d\bigr).
\end{equation}

Using Duhamel formulas as in the contraction argument, it follows the convergence 
 \begin{equation*} 
  u_{\varepsilon}\longrightarrow u\qquad\text{in }L^\infty((0,\tau), L^\infty(\mathbb R^d)), 
\end{equation*} 
where $\tau$ is independent on $\varepsilon$. Indeed
  \begin{equation*}
  \begin{aligned}
    u_\varepsilon(t)-u(t) &= G(t)\star \bigl(u_{0,\varepsilon}-u_0\bigr) \\ 
    &\qquad - \int_0^t \nabla G(t-r)\star \left[ b_\varepsilon(r)u_\varepsilon(r)^{1+k} - b(r)u(r)^{1+k} \right]dr.
  \end{aligned}
\end{equation*}

The thesis follows from similar computations as done before, observing that one can decompose 
\begin{equation*} 
\begin{aligned} 
  b_\varepsilon u_\varepsilon^{1+k} - bu^{1+k} &= b_\varepsilon \left( u_\varepsilon^{1+k}-u^{1+k} \right) + (b_\varepsilon-b)u^{1+k}.
\end{aligned} 
\end{equation*} 
and by
\begin{equation*}
  0\leq u_\varepsilon,u\leq C,
\end{equation*}
this $C$ only depends on $\norm{b}_\infty, \norm{u_0}_1, \norm{u_0}_\infty$. Then 
\begin{equation*}
   \left| u_\varepsilon^{1+k}-u^{1+k} \right| \leq (1+k) C^k|u_\varepsilon-u|. 
\end{equation*}

The same subdivision argument used in the contraction proof now gives 
\begin{equation*} u_\varepsilon\longrightarrow u \qquad\text{in } C([0,T];L^1(\mathbb R^d)). 
\end{equation*} 
We also observe that the function $m_{u_\varepsilon(t)}$ is converging to $m_{u(t)}$ uniformly on all the times when $u_\varepsilon,u$ are well defined.
Indeed, in general, for every non-negative $f,g\in L^1(\mathbb R^d)$, 
\begin{equation*} 
      m_f(s)-m_g(s) \leq \sup_{|A|=s} \int_A(f-g)\,dx \leq \|f-g\|_1. 
\end{equation*} 
We can summarise the reduction argument as follows.

\begin{proposition}
  Every solution $u$ of Proposition $\ref{prop:local}$ can be approximated with a sequence of function $u_\varepsilon$ such that 
  \begin{equation*}
    u_\varepsilon \longrightarrow\ u \text{ in norms } C([0,T^*),L^1(\R^d)) \text{ and } L^\infty((0,T^*),L^\infty(\R^d))
  \end{equation*}
  Moreover every estimate on $m_\varepsilon\le C$ pass to the limit as $m_\varepsilon\le C$ with the same constant.
\end{proposition}

\section{Global Existence}
\label{sec:globex}
We now study the global existence in the subcritical and critical cases. As established in Subsection \ref{subsec:smoothreduction}, we can assume the solution to be smooth. Denote 
\begin{equation*}
    \Omega_s(t):=\{x \in \mathbb R^d: u(t,x)>u^*(t,s)\}.
\end{equation*}

\subsection{Entropy differential inequality}
By Lemma \ref{lem:exact-set}, define the function $$m(t,s):=\int_0^s u^*(t,z)\ dz$$
\begin{lemma}\label{lem:time}
For every $s$, and any maximizer $A_{t,s}$ of $m(t,s)$, the map $t\mapsto m(t,s)$ is absolutely continuous and it holds for almost every $t$,
\begin{equation}
 \partial_t m(t,s)=\int_{A_{t,s}}\partial_t u(t,x)\,d x.
\end{equation}
\end{lemma}

\begin{proof}
By Lemma \ref{lem:exact-set},
\begin{equation*}
    |m(t_1,s)-m(t_2,s)|\le \|u(t_1)-u(t_2)\|_1,
\end{equation*}
thus, the map $t \mapsto m(t,s)$ is absolutely continuous.
Again using Lemma \ref{lem:exact-set}, for every maximizer set $A_{t,s}$ at time $t$ for $m(t,s)$,
\begin{equation*}
\frac{m(t+h,s)-m(t,s)}{h} \geq\int_{A_{t,s}} \frac{u(t+h)-u(t)}{h}\, d x.
\end{equation*}
and 
\begin{equation*}
\frac{m(t,s)-m(t-h,s)}{h} \le\int_{A_{t,s}} \frac{u(t)-u(t-h)}{h}\, d x.
\end{equation*}
The thesis follows by sending $h\to0^+$.
\end{proof}
\begin{lemma}\label{lem:smooth-maximizer}
Fix $t>0$. For a.e. $s>0$ such that
\begin{equation*}
    m_{ss}(t,s)<0,
\end{equation*}
the value $u^*(t,s)$ is a regular value of $u(t,\cdot)$.
Consequently,
\begin{equation*}
    \Omega_s(t):=
    \{x\in\R^d:u(t,x)>u^*(t,s)\}
\end{equation*}
satisfies
\begin{equation*}
    |\Omega_s(t)|=s
\end{equation*}
and has smooth boundary
\begin{equation*}
    \partial\Omega_s(t)
    =
    \{x\in\R^d:u(t,x)=u^*(t,s)\}.
\end{equation*}
Moreover, every maximizing set $A$ of measure $s$ coincides
with $\Omega_s(t)$ up to a negligible set.
\end{lemma}

\begin{proof}
Note that $m_{ss}(t,s)=d/ds \ u^*(t,s).$ Let $\mathcal C_t$ denote the set of critical values of
$u(t,\cdot)$. By Sard's theorem \cite{morse1939behavior}\cite{sard1942measure},
\begin{equation*}
    |\mathcal C_t|=0.
\end{equation*}
The coarea formula
\begin{equation*}
    \int_{\{s:u^*(t,s)\in\mathcal C_t\}}|m_{ss(t,s)}|\,ds=0.
\end{equation*}
Therefore, for a.e. $s$ such that $m_{ss(t,s)}<0$, one has
$u^*(t,s)\notin\mathcal C_t$. Thus $u^*(t,s)$ is a regular value, and
the regular level-set theorem implies that
\begin{equation*}
    \{u(t,\cdot)=u^*(t,s)\}
\end{equation*}
is a smooth hypersurface, and
\begin{equation*}
    |\{u(t,\cdot)=u^*(t,s)\}|=0.
\end{equation*}
\end{proof}

\begin{lemma}\label{lem:diffusion}
	For almost every $s$, given any maximizer $A_{t,s}$, it holds
	\begin{equation}\label{eq:diffusion-bound}
		\int_{A_{t,s}}\Delta u\,d x
		\leq d^{2-2/d}\omega_d^{2/d} s^{2-2/d}m_{ss}(t,s).
	\end{equation}
\end{lemma}

\begin{proof}
The proof is divided in two cases, $m_{ss}<0$ or $m_{ss}=0$. Assume that $m_{ss}=0$. By Lemma \ref{lemma:stamp} applied to $u(t)-u^*(t,s)$, it holds
    
\begin{equation*}
    \Delta u=0 \qquad\text{a.e. on }\{u=u^*(t,s)\}.
\end{equation*}
Hence
\begin{equation*}
    \int_{A_{t,s}}\Delta u\,dx
    =
    \int_{\Omega_s}\Delta u\,dx.
\end{equation*}
Let $\chi_\varepsilon$ be smooth and nondecreasing approximation of $\mathbf 1_{(0,\infty)}$. Then
\begin{equation*}
\begin{aligned}
    \int_{\mathbb R^d}
    \chi_\varepsilon(u-u^*(t,s))\Delta u\,dx
    &=
    -\int_{\mathbb R^d}
    \chi_\varepsilon'(u-u^*(t,s))|\nabla u|^2\,dx\\
    &\leq0.
\end{aligned}
\end{equation*}
Passing to the limit gives
\begin{equation*}
    \int_{\{u>q\}}\Delta u\,dx\leq0.
\end{equation*}
The thesis follows in this case. Assume now $m_ss<0$. By Lemma \ref{lem:smooth-maximizer}, and Gauss-Green, we can write
\begin{equation*}
  \int_{A_{t,s}}\Delta u\,d x
  = \int_{\Omega_s(t)}\Delta u\,d x=-\int_{\partial \Omega_s(t)}|\nabla u|\,d\mathcal H^{d-1}.
\end{equation*}
Using Cauchy-Schwarz, we can write
\begin{equation*}
	\begin{split}
		(\mathcal H^{d-1}(\partial \Omega_s))^2&\le
		\int_{\partial \Omega_s}|\nabla u|\,d\mathcal H^{d-1}
		\cdot\int_{\partial \Omega_s}\frac1{|\nabla u|}
		\,d\mathcal H^{d-1}\\
		&= -\int_{\Omega_s}\Delta u\,d x
		\cdot\int_{\partial \Omega_s}\frac1{|\nabla u|}
		\,d\mathcal H^{d-1}
	\end{split}
\end{equation*}
therefore from \eqref{eq:Mss-level} 
\begin{equation*}
  \int_{\Omega_s(t)}\Delta u\,d x\leq \Per^2(\Omega_s(t)) m_{ss}.
\end{equation*}
As $m_{ss}<0$, the isoperimetric inequality \cite{maggi2012sets} yields
\begin{equation*}
  \int_{\Omega_s(t)}\Delta u\,d x\le \Per^2(\Omega_s(t))m_{ss}
  \leq d^{2-2/d}\omega_d^{2/d}s^{2-2/d}m_{ss}.
\end{equation*}
This concludes the proof.
\end{proof}

\begin{lemma}\label{lem:drift}
For almost every $s$, given any maximizer $A_{t,s}$, it holds
\begin{equation}\label{eq:drift-bound}
  -\int_{A_{t,s}}\dive\bigl(bu^{1+k}\bigr)\,d x
  \leq p'\norm{ (\dive b(t))_-}_{L^{p,\infty}}
  s^{1-1/p}(m_s)^{1+k}.
\end{equation}
\end{lemma}

\begin{proof}
By characterization of the maximizer, 
\begin{equation*}A_{t,s}=\{u>u^*(t,s)\}\cup B,
\qquad
B\subseteq\{u=u^*(t,s)\}.
\end{equation*}
Define
\begin{equation*}
\Psi_{u^*(t,s)}(r):=(r^{1+k}-(u^*(t,s))^{1+k})_+.
\end{equation*}
By the chain rule,
\begin{equation*}
\nabla\Psi_{u^*(t,s)}(u)
=
(1+k)u^k\nabla u\,\mathbf 1_{\{u>_{u^*(t,s)}\}}.
\end{equation*}
Integrating over $\mathbb R^d$ gives
\begin{equation*}
\int_{\{u>u^*(t,s)\}}
\dive\!\left(bu^{1+k}\right)\,dx
=
(u^*(t,s))^{1+k}
\int_{\{u>u^*(t,s)\}}\dive b\,dx.
\end{equation*}
On the other hand, by Lemma \ref{lemma:stamp},
\begin{equation*}
\nabla u=0
\qquad\text{a.e. on }\{u=r\},
\end{equation*}
for every $r\in R$. Hence, a.e. on $B$,
\begin{equation*}
\dive\!\left(bu^{1+k}\right)
=
u^{1+k}\dive b
+(1+k)u^k b\cdot\nabla u
=
u^{1+k}\dive b.
\end{equation*}
Combining the two identities yields
\begin{equation*}
\int_{A_{t,s}}
\dive\!\left(bu^{1+k}\right)\,dx
=
(u^*(t,s))^{1+k}
\int_{A_{t,s}}\dive b\,dx.
\end{equation*}
Consequently,
\begin{equation*}
-\int_{A_{t,s}}
\dive\!\left(bu^{1+k}\right)\,dx
\leq
(u^*(t,s))^{1+k}
\int_{A_{t,s}}(\dive b)_-\,dx.
\end{equation*}

The thesis follow from Theorem \ref{Theo:holder_lorentz}, since $\|1_{A_{t,s}}\|_{L^{p',1}}= p' |A_{t,s}|^{1/p'}= p' s^{1/p'}.$

\end{proof}

Combining the Lemmas \ref{lem:time}, \ref{lem:diffusion} and \ref{lem:drift} it holds the following differential inequality.

\begin{proposition}
For almost every $s$, it holds
\begin{equation}\label{eq:concentration-raw}
  \partial_t m
  \leq d^{2-2/d}\omega_d^{2/d}s^{2-2/d}m_{ss}
  +p'\norm{ {(\dive b(t))_-}}_{L^{p,\infty}}
   s^{1-1/p}(m_s)^{1+k}.
\end{equation}

\end{proposition}

\subsection{Stationary solution}
Denote $a_d=d^{2-2/d}\omega_d^{2/d}$ and $a_p=a_p(T)=\sup_{t\in [0,T]}p'\norm{ {(\dive b)_-(t)}}_{L^{p,\infty}}$
We seek a stationary solution $\Phi$ for the equation
\begin{equation}\label{eq:stationary-equation}
  a_d s^{2-2/d}\Phi''
  +a_ps^{1-1/p}(\Phi')^{1+k}=0.
\end{equation}

\begin{proposition}
 For any $\lambda>0$, define
\begin{equation}\label{eq:profile-slope}
  \varphi_\lambda(z):=
  \left(
    \lambda+\frac{k a_p}{a_dk_c}z^{k_c}
  \right)^{-1/k},
  \qquad z\geq0,
\end{equation}
and
\begin{equation}\label{eq:profile-primitive}
  \Phi_{\lambda}(s):=\int_0^s\varphi_\lambda(z)\ dz.
\end{equation}
Then $\Phi_{\lambda}$ solves \eqref{eq:stationary-equation} on $(0,\infty)$. 
\end{proposition}

\begin{proof}
It holds
\begin{equation*}
  \varphi_\lambda'(z)
  =-\frac {a_p}{a_d}z^{k_c-1}\varphi_\lambda(z)^{1+k},
\end{equation*}
and
\begin{equation*}
  \left(2-\frac2d\right)+(k_c-1)
  =1-\frac1p,
\end{equation*}
hence
\begin{equation*}
  a_d z^{2-2/d}\varphi_\lambda'
  +a_p z^{1-1/p}\varphi_\lambda^{1+k}=0.
\end{equation*}
Using $\Phi_\lambda'=\varphi_\lambda$, the result follows.
\end{proof}
\begin{lemma}\label{lem:barlambda}
 
Let $\bar s=\frac{\|u(0)\|_1}{\|u(0)\|_\infty}$, $A:=\dfrac{k\,a_p(T)}{a_d\,k_c}$, and define
\begin{equation}\label{eq:uniformconst}
\bar\lambda_T:=
\begin{cases}
A\,\bar s^{\,k_c}\exp\ \Big(-k_c\,(2A)^{1/k_c}\,\|u_0\|_1\Big),
& k=k_c,\\[2ex]
\min\left\{
A\left(\dfrac{k}{2(k_c-k)\,(2A)^{1/k}\,\|u_0\|_1}\right)^{\frac{k\,k_c}{k_c-k}},\;
A\left(2^{-\frac{k}{k_c-k}}\,\bar s\right)^{k_c}
\right\},
& 0<k<k_c,
\end{cases}
\end{equation}
 It holds 
  \begin{equation*}
    \Phi_{\bar\lambda_T}\left(\bar s\right)\ge \|u(0)\|_1.
  \end{equation*}
\end{lemma}
\begin{proof}

Set $z_\lambda:=(\lambda/A)^{1/k_c}$, so that $\lambda\le A z^{k_c}$ for $z\ge z_\lambda$ and hence
\begin{equation}\label{eq:tail-lb}
\varphi_\lambda(z)=\big(\lambda+Az^{k_c}\big)^{-1/k}\ \ge\ (2A)^{-1/k}\,z^{-k_c/k},
\qquad z\ge z_\lambda .
\end{equation}
Assume $z_\lambda\le\bar s$. Integrating \eqref{eq:tail-lb} on $(z_\lambda,\bar s)$,
\begin{equation}\label{eq:tail-int}
\Phi_\lambda(\bar s)\ \ge\ (2A)^{-1/k}\int_{z_\lambda}^{\bar s} z^{-k_c/k}\,dz .
\end{equation}

\emph{Critical case $k=k_c$.} Equation \eqref{eq:tail-int} becomes
\begin{equation*}
\Phi_\lambda(\bar s)\ \ge\ (2A)^{-1/k_c}\log\frac{\bar s}{z_\lambda}
\ =\ \frac{(2A)^{-1/k_c}}{k_c}\,\log\frac{A\,\bar s^{\,k_c}}{\lambda},
\end{equation*}
using $z_\lambda^{k_c}=\lambda/A$. The right-hand side is $\ge \|u_0\|_1$ precisely when
$\lambda\le A\bar s^{\,k_c}\exp\big(-k_c(2A)^{1/k_c}\|u_0\|_1\big)$, which is the stated $\bar\lambda$.
This value satisfies the condition $z_{\bar\lambda}\le\bar s$, as required above.

\emph{Subcritical case $0<k<k_c$.} Now $k_c/k>1$, so \eqref{eq:tail-int} yields
\begin{equation*}
\Phi_\lambda(\bar s)\ \ge\ (2A)^{-1/k}\,\frac{k}{k_c-k}
\Big(z_\lambda^{-\frac{k_c-k}{k}}-\bar s^{\,-\frac{k_c-k}{k}}\Big).
\end{equation*}
If moreover $z_\lambda\le 2^{-\frac{k}{k_c-k}}\bar s$, then
$\bar s^{-\frac{k_c-k}{k}}\le\frac12 z_\lambda^{-\frac{k_c-k}{k}}$ and therefore
\begin{equation*}
\Phi_\lambda(\bar s)\ \ge\ (2A)^{-1/k}\,\frac{k}{2(k_c-k)}\,z_\lambda^{-\frac{k_c-k}{k}} .
\end{equation*}
So the condition $\Phi_\lambda(\bar s)\ge \|u_0\|_1$ is satisfied when it holds
\begin{equation*}
  z_\lambda^{\frac{k_c-k}{k}}\le\frac{k}{2(k_c-k)(2A)^{1/k}\|u_0\|_1},
\end{equation*} therefore, when
\begin{equation*}
\lambda\ =\ A z_\lambda^{k_c}\ \le\
A\left(\frac{k}{2(k_c-k)\,(2A)^{1/k}\,\|u_0\|_1}\right)^{\frac{k\,k_c}{k_c-k}} .
\end{equation*}
Taking $\bar\lambda$ to be the minimum of this quantity and
$A\big(2^{-\frac{k}{k_c-k}}\bar s\big)^{k_c}$ ensures both this inequality and the restriction
$z_{\bar\lambda}\le 2^{-\frac{k}{k_c-k}}\bar s$ used above, which proves the claim.
\end{proof}
\begin{lemma}\label{lem:ineqPhim}
It holds
\begin{equation*}
    \Phi_{\bar\lambda_T}(s)\ge m(t,s)
\end{equation*}
for every $t\in[0, T]$, $T<T^*$. 
\end{lemma}

\begin{proof}
By Lemma \ref{lem:barlambda},
\begin{equation*}
    \Phi_{\bar\lambda_T}(\bar s)\ge \|u(0)\|_1.
\end{equation*}
It is easy to prove that
\begin{equation*}
    m(0,s)\le \min\{\|u(0)\|_\infty s, \|u(0)\|_1\}.
\end{equation*}
Also, since $\Phi_{\bar\lambda_T}$ is increasing and concave, $s\mapsto\Phi_{\bar\lambda_T}(s)/s$ is decreasing. For $s\le\bar s$, 
\begin{equation*}
    \frac{\Phi_{\bar\lambda_T}(s)}{s}\ge\frac{\Phi_{\bar\lambda_T}(\bar s)}{\bar s}\ge \frac{\|u(0)\|_1}{\bar s}={\|u(0)\|_\infty},
\end{equation*}
so that
\begin{equation*}
    \Phi_{\bar\lambda_T}(s)\ge \|u(0)\|_\infty s\ge m(0,s).
\end{equation*}
For $s>\bar s$, and every $t\ge 0$,
\begin{equation*}
    \Phi_{\bar\lambda_T}(s)\ge \Phi_{\bar\lambda_T}(\bar s) \ge \|u(0)\|_1=\|u(t)\|_1\ge m(t,s).
\end{equation*}
Now consider $v=(m-\Phi)^+$. For every $t<T<T^*$, by Mean Value Theorem, 
\begin{equation*}
  (m_s)^{1+k}-(\Phi')^{1+k}\le (1+k) K^k |v_s|,
\end{equation*}
where $K = \max \{ {\bar\lambda_T}^{-1/k}, \sup_{[0,T]} \|u(t)\|_\infty \}$. Using the inequality for $m$ and $\Phi$, it holds
\begin{equation*}
  \begin{split}
  \frac{d}{dt} \int_0^{\bar s} s^{2/d-2} v^2 &\le -a_d \int_0^{\bar s}  (v_s)^2+ C K^k \int_0^{\bar s} s^{k_c-1} |v_s|v\\
  &\le C K^k {\bar s}^{2/d-2/p}\int_0^{\bar s} s^{2/d-2} v^2
  \end{split}
\end{equation*}
and from Gronwall, $m(t,s)\le \Phi(s)$ for every $s\le \bar s$ and every time $t<T<T^*$. This concludes the proof.
\end{proof}

\subsection{Proof of global existence}

\begin{proof}[Proof of Theorem \ref{thm:global-main}]

For every $T<T^*$, using Lemma \ref{lem:ineqPhim} and l'Hôpital's rule, it holds
\begin{equation*}
  \norm{u(T)}_\infty=u^*(T,0)
  =\lim_{s\to0^+}\frac{m(t,s)}s
  \leq\lim_{s\to0^+}\frac{\Phi_{\bar\lambda_T}(s)}s
  =\lambda_T^{-1/k}.
\end{equation*}
In particular, if $T^*$ where finite, then the inequality
\begin{equation*}
  \norm{u(T)}_\infty\le\lambda_{T^*}^{-1/k}.
\end{equation*}
leads to a contradiction.
\end{proof}

\section{Blow-up}\label{sec:blowup}

We now deal with the supercritical case, where we will provide examples of blow-up in finite time. The result for the supercritical case generalizes the works \cite{toscani2012finite}, \cite[Section 5]{bianchini2024existence}. A similar construction can be used to find blow-up examples in the case $b\in L_t^\infty L_x^p$, $k>\frac{1}{d}-\frac{1}{p}$.
\subsection{Supercritical case}\label{ss:supcrit}
Assume condition $k>\frac{2}{d}-\frac{1}{p}$. Let $\alpha$ such that  $k d> 1 + \alpha > 2-d/p$ and consider a bounded function $b$ with $\dive_x b \in L^p$, possibly smooth out of the origin, such that
\begin{equation}
\label{eq:syst:b:2}
     |b(x)| \text{ is } \begin{cases}
         =|x|^\alpha &|x|\le \bar \rho\\
         \ge \beta (\bar \rho)^\alpha &|x|\ge \bar \rho
     \end{cases} , \quad  \beta\in\left[\left(\frac{dk-\alpha-1}{dk-1}\right)^\frac{k}{\alpha},1\right).
\end{equation}
and we impose $x \cdot b= -|x||b|$, i.e. b of the kind $b=-w(|x|)x$.

\begin{proposition}
\label{lemma:func1:2}
For every measurable positive function $f: \R^d \to [0,\infty)$, the following holds. 
\begin{enumerate}
    \item If
    \begin{equation}
    \label{eq:cond:1:c:2}
    \begin{split}
        \left(\frac{2(k+1)}{dk-\alpha-1}\right)^{\frac{k+1}{(d+2)k-\alpha+1}}  \left( \frac{w_d k}{k - (\alpha + 1)} \right)^{-\frac{k}{(d+2)k + 1 - \alpha}} \cdot \\ \cdot\left( \int f^{k+1} |xb| dx \right)^{- \frac{1} {(d+2)k + 1-\alpha}} E^{\frac{k+1}{(d+2)k + 1 - \alpha}}\le \bar\rho
        \end{split}
    \end{equation}
    then
    \begin{equation*}
    \begin{split}
        \|f\|_1 &\le C(k,d,\alpha)\left( \int f^{k+1} |xb| dx \right)^{\frac{2}{(2+d)k + 1 - \alpha}} E^{\frac{dk - \alpha - 1}{(d+2)k + 1 - \alpha}}
     \end{split}
    \end{equation*}
    \item If
    \begin{equation}
    \label{eq:cond:2:c:2}
    \begin{split}
        \left(\frac{1}{\beta\bar\rho}\right)^{-\frac{\alpha(k+1)}{((d+2)k+1)k}} \left(\frac{2(k+1)}{dk-1}\right)^{\frac{k+1}{(d+2)k+1}}\left( \frac{w_d k}{k -  1} \right)^{-\frac{k}{(d+2)k + 1 }} \cdot \\ \cdot\left( \int f^{k+1} |xb| dx \right)^{- \frac{1} {(d+2)k + 1}} E^{\frac{k+1}{(d+2)k + 1}} \geq \bar \rho,
        \end{split}
    \end{equation}
    then 
    \begin{equation*}
    \begin{split}
        \|f\|_1 &\le C(k,d,\alpha)\left(\frac{1}{\beta\bar\rho}\right)^{\frac{2\alpha(k+1)}{((d+2)k+1)k}}\left( \int f^{k+1} |xb| dx \right)^{\frac{2}{(2+d)k + 1 }} E^{\frac{dk - 1}{(d+2)k + 1 }}.
     \end{split}
    \end{equation*}
\end{enumerate}
\end{proposition}

\begin{proof} {\it Case 1.} Define
\begin{equation}
    \label{Equa:def_R'}
    \begin{split}
R' :=  \left(\frac{2(k+1)}{dk-\alpha-1}\right)^{\frac{k+1}{(d+2)k-\alpha+1}}\left( \frac{w_d k}{k - (\alpha + 1)} \right)^{-\frac{k}{(d+2)k + 1 - \alpha}} \cdot \\ \cdot\left( \int f^{k+1} |xb| dx \right)^{- \frac{1} {(d+2)k + 1-\alpha}} E^{\frac{k+1}{(d+2)k + 1 - \alpha}},
 \end{split}
\end{equation}
and compute by H\"older inequality
\begin{equation*}
    \begin{split}
        m &= \int f dx = \int_{|x| \leq R'} f dx + \int_{|x| > R'} f dx \\
        &\le \left( \int f^{k+1} |x|^{\alpha + 1} dx \right)^{\frac{1}{k+1}} \left( \int_{|x| \leq R'} \frac{1}{|x|^{\frac{1+\alpha}{k}}} dx \right)^{\frac{k}{k+1}} + \frac{1}{{R'}^2} E \\
    \big[ (\ref{eq:cond:1:c:2}) \big] \quad   &= \left( \int f^{k+1} |xb| dx \right)^{\frac{1}{k+1}} \left( \frac{\omega_d k}{dk-(\alpha+1)} \right)^{\frac{k}{k+1}} {(R')}^{\frac{dk-(\alpha+1)}{k+1}} + \frac{1}{{(R')}^2} E \\
     \big[   (\ref{Equa:def_R'}) \big] \quad  &= 
     4^{-\frac {k+1} {(d+2) k - \alpha+1}}\left( \frac{(d+2) k - \alpha + 1}{k + 1} \right)\left( \frac{k + 1}  {d k - \alpha - 1}\right)^{\frac{d k  - \alpha- 1 }{(d+2)k-\alpha +1}} \\
     &\qquad\cdot\left( \frac{\omega_d k}{dk-(\alpha+1)} \right)^{\frac{2k}{(d+2)k-\alpha+1}}\left( \int f^{k+1} |xb| dx \right)^{\frac{2}{(2+d)k + 1 - \alpha}} E^{\frac{dk - \alpha - 1}{(d+2)k + 1 - \alpha}},
    \end{split}
\end{equation*}
    which is the first estimate.
    
\noindent{\it Case 2.}
Define
\begin{equation}
\label{Equa:def_N}
\begin{split}
N := \left(\frac{1}{\beta\bar\rho}\right)^{-\frac{\alpha(k+1)}{((d+2)k+1)k}} \left(\frac{2(k+1)}{dk-1}\right)^{\frac{k+1}{(d+2)k+1}}\left( \frac{w_d k}{k -  1} \right)^{-\frac{k}{(d+2)k + 1 }}\cdot \\ \cdot \left( \int f^{k+1} |xb| dx \right)^{- \frac{1} {(d+2)k + 1}} E^{\frac{k+1}{(d+2)k + 1}},
\end{split}
\end{equation}
and compute
\begin{equation*}
    \begin{split}
        m &= \int f dx = \int_{|x| \leq N} f dx + \int_{|x| > N} f dx \\
        &\le \left( \int f^{k+1} |xb| dx \right)^{\frac{1}{k+1}} \left( \int_{|x| \leq N} \frac{1}{|xb|^{\frac{1}{k}}} dx \right)^{\frac{k}{k + 1}} + \frac{1}{(N)^2} E \\
     \big[ (\ref{eq:syst:b:2}), (\ref{eq:cond:2:c:2}) \big] \quad  &\le \left( \int f^{k+1} |xb| dx \right)^{\frac{1}{k+1}} \left(  \int_{|x|\le \bar \rho} \frac{1}{|x|^{\frac{1 + \alpha}{k}}} dx +  \int_{\bar \rho\le |x|\le N} \frac{1}{|x  (\beta \bar \rho)^\alpha|^{\frac{1}{k}}} dx \right)^{\frac{k}{k+1}} \\
 &\qquad+  \frac{1}{(N)^2} E \\       
 \quad        &\le \left( \int f^{k+1} |xb| dx \right)^{\frac{1}{k+1}} \left( \frac{\omega_d k}{dk - 1} \right)^{\frac{k}{k+1}}(N)^{\frac{d k-1}{k+1}} \frac{1}{(\beta\bar\rho)^\frac{\alpha}{k}}+\frac{1}{(N)^2} E \\
  \big[ \eqref{Equa:def_N} \big] &=4^{-\frac {k+1} {(d+2) k -+1}}\left( \frac{(d+2) k + 1}{k + 1} \right)\left( \frac{k + 1}  {d k - 1}\right)^{\frac{d k  - 1 }{(d+2)k +1}} \left( \frac{\omega_d k}{dk-1} \right)^{\frac{2k}{(d+2)k+1}}\\
     &\qquad\cdot\left(\frac{1}{\beta\bar\rho}\right)^{\frac{2\alpha(k+1)}{((d+2)k+1)k}}\left( \int f^{k+1} |xb| dx \right)^{\frac{2}{(2+d)k + 1 }} E^{\frac{dk - 1}{(d+2)k + 1 }}, 
    \end{split}
\end{equation*}
This is the second estimate in the statement.
\end{proof}

\begin{proof}[Proof of Theorem \ref{thm:blowup-main}]
The choice of the constant $\bar\rho$ covers both the cases \eqref{eq:cond:1:c:2} and \eqref{eq:cond:2:c:2}, as we can see from the conditions
\begin{equation*}
           \left(\int f^{k+1} |xb| dx \right)^{-1} E^{{k+1}}\le \bar\rho^{(d+2)k + 1 - \alpha} \left(\frac{2(k+1)}{dk-\alpha-1}\right)^{-k-1}\left( \frac{w_d k}{k - (\alpha + 1)} \right)^{k} 
\end{equation*}
and 
\begin{equation*}
           \left(\int f^{k+1} |xb| dx \right)^{-1} E^{{k+1}}\ge  \left(\frac{1}{\beta\bar\rho}\right)^{\frac{\alpha(k+1)}{k}} \bar\rho^{(d+2)k + 1} \left(\frac{2(k+1)}{dk-1}\right)^{-k-1}\left( \frac{w_d k}{k - 1} \right)^{k}.
\end{equation*}

By Proposition \ref{lemma:func1:2} and \eqref{eq:pde}, one obtains
\begin{equation*}
    \frac{dE}{dt} \leq 2d\|u_0\|_1 - C \min \bigg\{ {E^{-\frac{dk - \alpha - 1}{2}}}, E^{-\frac{dk-1}{2}} \bigg\},
    \end{equation*}
which leads in finite time to $E(t) \to 0$ if $E(u_0)\ll 1$.
Clearly $E(t) \to 0$ and finite mass imply that $\|u(t)\|_\infty$ blows up.
\end{proof}

\section{Long-time behaviour of solutions}
\label{sec:longbehav}

In this last section, we discuss the long-time behaviour of the solutions when there are uniform bounds on the drift $b$.
\begin{equation*}
(\dive b)_-\in L^\infty
  \bigl((0,\infty),L^{p,\infty}(\R^d)\bigr),L^\infty(\R^3))
\end{equation*}

We prove the following theorem.

\begin{theorem}
    \label{Theo:long_time_1}
    Assume condition \eqref{eq:assumptions2}. The following holds.
    \begin{enumerate}
        \item \label{Point_1:long_time_1} If $k <k_c$, there are stationary solutions.

        \item \label{Point_2:long_time_1} If $k =k_c$, then every well defined solution $u$ decays like $\|u(t)\|_\infty\leq \mathcal O(t^{-\frac{d}{2}}) $.
        \item \label{Point_3:long_time_1} If $k>k_c$, and the solution $u$ is uniformly bounded in time and space, i.e. $u \in L^\infty_{t,x}$, then decays like $\|u(t)\|_\infty\leq \mathcal O(t^{-\frac{d}{2}}) $.
    \end{enumerate}
\end{theorem}

\begin{proof} 
{\it Point \eqref{Point_1:long_time_1}.} It is easy to prove that 
\begin{equation*}
    u(x) = (1 + |x|^2)^{-\alpha}
\end{equation*}
is a stationary solution for the drift
\begin{equation*}
    b(x) = -  2 \alpha x(1 + |x|^2)^{\alpha k-1}.
\end{equation*}
It is enough to take $\alpha=(2-d/p)/(2k)$ to have $\dive  b\in L^{p,\infty}$.
\medskip

{\it Point \eqref{Point_2:long_time_1}.} 
We look for self-similar solutions for the differential inequality
\begin{equation*}
\partial_t m \leq a_d s^{2-\frac2d}m_{ss}+a_p s^{1-1/p} (m_s)^{1+k},
\end{equation*}
of the kind
\begin{equation*}
\hat m_\varepsilon(t,s)=\bar m \Phi\left( \frac{s}{(t+\varepsilon)^{d/2}}\right).
\end{equation*}
Notice that in this case $a_p$ is time-independent. Fix $\bar m>m$ and look for a self-similar supersolution of the form
\begin{equation*}
\hat m_\varepsilon(t,s)
=
\bar m \Phi\left(
\frac{s}{(t+\varepsilon)^{d/2}}
\right).
\end{equation*}
For
\begin{equation*}
z=\frac{s}{(t+\varepsilon)^{d/2}},
\qquad
\varphi(z)=\Phi'(z),
\end{equation*}
the problem reduces to solving the equation
\begin{equation}
\varphi'+\frac{a_p\bar m^k}{a_d}z^{k-1}\varphi^{1+k}+\frac{d}{2a_d}z^{2/d-1}\varphi=0.
\label{eq:profile-ode}
\end{equation}
With standard methods, we find the family of solutions
\begin{equation}
\varphi_\lambda(z)=\exp\left(-\frac{d^2}{4a_d}z^{d/2}\right)\left[\lambda+\frac{ka_p\bar m^k}{a_d}\int_0^z\xi^{k-1}\exp\left(-\frac{kd^2}{4a_d }\xi^{2/d}\right)d\xi\right]^{-1/k}.
\label{eq:critical-profile}
\end{equation}

Notice that for every $\lambda$, it holds $\varphi_\lambda\in L^1.$
Moreover, it is easy to prove that
\begin{equation*}
\lim_{\lambda\to 0^+}\int_0^\infty \varphi_\lambda(z)\,dz=+\infty
\end{equation*}
And
\begin{equation*}
\lim_{\lambda\to =+\infty}\int_0^\infty \varphi_\lambda(z)\,dz
=0.
\end{equation*}
Therefore, there exists $\bar\lambda >0$ such that
\begin{equation*}
\int_0^\infty \varphi_{\bar\lambda}(z)\,dz=1.
\end{equation*}
Define
\begin{equation*}
\Phi(s):=\int_0^s \varphi_{\bar\lambda}(z)\,dz.
\end{equation*}
Then
\begin{equation*}
\Phi(0)=0,
\qquad
\lim_{z\to \infty}\Phi(z)=1,
\qquad
\Phi''\leq0.
\end{equation*}

Since $\lim_{z\to \infty}\bar m\Phi(z)=\bar m>\|u(0)\|_1$, for $\varepsilon\ll 1$ it holds
\begin{equation*}
\hat m_\varepsilon(0,\bar s)\geq m.
\end{equation*}
At this point, the function $m_\varepsilon(t,s)$ has the properties of $ \Phi_\lambda(s)$ and exactly with the same computations of Lemma \ref{lem:ineqPhim}, one can prove that 
\begin{equation*}
m(t,s)
\leq
\hat m_\varepsilon(t,s).
\end{equation*}
Finally, for $\varepsilon\to 0^+$, we obtain
\begin{equation*}
\|u(t)\|_\infty\le
\lim_{s\to0^+}\frac{m(t,s)}{s}\leq
\bar m\,t^{-d/2}\Phi'(0)=\bar m\,\bar\lambda^{-1/k}t^{-d/2}.
\end{equation*}
Since $\bar m$, $\bar \lambda$ and $k$ do not depend on the time $t$, the proof is completed.
{\it Point \eqref{Point_3:long_time_1}.} The proof is almost totally the same as Point \eqref{Point_2:long_time_1}. The only difference is that, one has to estimate the drift term as
\begin{align*}
-\int_{\Omega_s}\dive\bigl(bu^{1+k}\bigr)\,d x
  &=-(u^*(t,s))^{1+k}
    \int_{\partial \Omega_s}b\cdot n\,d\mathcal H^{d-1}\\
  &=-(u^*(t,s))^{1+k}\int_{\Omega_s}\dive b\,\\
  &\leq p'\norm{ {(\dive b(t))_-}}_{L^{p,\infty}}\norm{u}_{L^\infty_{t,x}}^{k-k_c}s^{1-1/p}(m_s)^{1+k_c}
\end{align*}
\end{proof}
\printbibliography

\end{document}